\documentclass[12pt,draft,reqno]{amsart}
\usepackage{array,latexsym}

\usepackage{amsmath,amssymb,amscd,amsfonts,amsthm}

\usepackage{paralist}

\usepackage{mathrsfs}
\usepackage{dsfont}
\usepackage{bbold}
\usepackage{mathbbol}
\usepackage[all]{xy}
\usepackage{enumerate}

\makeatletter
\renewcommand*\env@matrix[1][\arraystretch]{%
  \edef\arraystretch{#1}%
  \hskip -\arraycolsep
  \let\@ifnextchar\new@ifnextchar
  \array{*\c@MaxMatrixCols c}}
\makeatother

{\catcode`\@=11
\gdef\n@te#1#2{\leavevmode\vadjust{%
 {\setbox\z@\hbox to\z@{\strut#1}%
  \setbox\z@\hbox{\raise\dp\strutbox\box\z@}\ht\z@=\z@\dp\z@=\z@%
  #2\box\z@}}}
\gdef\leftnote#1{\n@te{\hss#1\quad}{}}
\gdef\rightnote#1{\n@te{\quad\kern-\leftskip#1\hss}{\moveright\hsize}}
\gdef\?{\FN@\qumark}
\gdef\qumark{\ifx\next"\DN@"##1"{\leftnote{\rm##1}}\else
 \DN@{\leftnote{\rm??}}\fi{\rm??}\next@}}
\DeclareFontFamily{OT1}{wncyr}{\hyphenchar\font45
}
\DeclareFontShape{OT1}{wncyr}{m}{n}{%
   <5> <6> <7> <8> <9> gen * wncyr
   <10> <10.95> <12> <14.4> <17.28> <20.74>  <24.88>wncyr10}{}
\DeclareFontShape{OT1}{wncyr}{m}{it}{%
   <5> <6> <7> <8> <9> gen * wncyi
   <10> <10.95> <12> <14.4> <17.28> <20.74> <24.88> wncyi10}{}
\DeclareFontShape{OT1}{wncyr}{m}{sc}{%
   <5> <6> <7> <8> <9> <10> <10.95> <12> <14.4>
   <17.28> <20.74> <24.88>wncysc10}{}
\DeclareFontShape{OT1}{wncyr}{b}{n}{%
   <5> <6> <7> <8> <9> gen * wncyb
   <10> <10.95> <12> <14.4> <17.28> <20.74> <24.88>wncyb10}{}
\input cyracc.def

\DeclareMathSizes{7}{7}{5}{3}

\theoremstyle{plain}

\newtheorem{theorem}{Theorem}

\newtheorem{proposition}{Proposition}
\newtheorem{lemma}{Lemma}

\newtheorem{corollary}{Corollary}

\theoremstyle{definition}

\newtheorem{definition}{Definition}

\newtheorem{nothing*}[theorem]{}
\newtheorem{subnothing*}[sub]{}

\theoremstyle{remark}

\newtheorem*{remarks}{Remarks}

\def\cd{\hskip -.5mm\cdot\hskip -.5mm}

\newcommand{\ds}{{\raisebox{-.7\height}{$\cdot$}}\hskip -.65mm
{\raisebox{-.17\height}{$\cdot$}}\hskip -.65mm
{\raisebox{0.36\height}{$\cdot$}}\hskip -.65mm
{\raisebox{.9\height}{$\cdot$}}}

\begin{document}

\title[Finitely many orbits in orbit closures]{Algebraic groups whose orbit closures\\
contain only finitely many orbits}

\author[Vladimir  L. Popov]{Vladimir  L. Popov}
\address{Steklov Mathematical Institute,
Russian Academy of Sciences, Gub\-kina 8,
Moscow 119991, Russia}
\email{popovvl@mi-ras.ru}

\maketitle

\begin{abstract}
We explore connected affine algebraic groups $G$, which enjoy the following finiteness property $\rm (F)$: for every algebraic action of $G$, the closure of every $G$-orbit contains only finitely many $G$-orbits.
We obtain two main
results. First, we classify such groups. Namely, we prove that a connected affine algebraic group $G$ enjoys property $\rm (F)$ if and only if $G$ is either a torus or a product of a torus and a one-dimensional connected unipotent algebraic group.
Secondly, we obtain a characterization of such groups in terms of the modality of action in the sense of V. Arnol'd. Namely, we prove that a connected affine algebraic group $G$ enjoys property $\rm (F)$ if and only if for every irreducible algebraic variety $X$ endowed with an algebraic action of $G$, the modality of $X$ is equal to
$\dim X-\max_{x\in X} G\cd x$.
\end{abstract}

\

\vskip 2mm

\

\section{\bf 1.\;Introduction}
The phenomenon of
finiteness of the sets of orbits in
orbit closures of algebraic actions of a connected affine
algebraic group $G$, which
arises  under certain restrictions on $G$ and actions, has been known for a long time and plays an essential role in several mathematical theories.
Since every $G$-orbit is dense and open in its closure,
it is the phenomenon of finiteness of the sets of $G$-orbits in algebraic $G$-varieties containing dense open $G$-orbit or, saying differently,
in equivariant open embeddings of algebraic homogeneous spaces of $G$.

The first example is that
of a torus $G$. In this case, every $G$-orbit closure contains only finitely many $G$-orbits. This is a
key
fact
of the theory of toric embeddings \cite{TE73}
(see, e.g., also
\cite{Fu93}).

Histori\-cal\-ly, the next example, which generalizes the previous one,  is the class of all equivariant open embeddings of a fixed algebraic homo\-ge\-neous space $
G/H$ of a connected complex reductive group $G$: by  \cite{A85}, every such embed\-ding contains only finitely many $G$-orbits if and only if  $H$
is a spherical subgroup of $G$.\;This fact is
a key
ingredient of the theory of spherical embeddings \cite{LV83}
(see, e.g., also
\cite{Ti11}).\;The earliest manifestation of this example is the case of parabolic $H$: then every equivariant embedding of $G/H$ coincides with $G/H$.

One more example is obtained if $G$ is uni\-po\-tent: in this case, every quasiaffine $G$-orbit closure coincides with this orbit \cite[Thm.\;2]{Ro61}. This is an important fact of the algeb\-raic transformation group theory.

In the present paper, we consider the
absolute
case, i.e.,
that,
in which no conditions on
actions under consideration are imposed.\;Namely, we
explore
con\-nec\-ted affine algebraic groups $G$ such that for every algebraic $G$-action,  every
$G$-orbit closure contains only finitely many $G$-orbits.
The
ex\-amp\-le of tori shows
that
such groups do exist.

First, we solve the classification problem for such groups.
The answer we found turned out to be rather unexpected for us. Namely,
our first main result
 is the following theorem (which is  a part of Theorem \ref{mcri} below).

\begin{theorem} \label{t1} Let $G$ be a connected affine algebraic group. The following properties are equivalent:
\begin{enumerate}[\hskip .9mm\rm(i)]
\item every $G$-orbit closure
of every algebraic action of $G$ contains only finitely many $G$-orbits;

 \item $G$ is either a torus or a direct product of a torus and a connected one-dimensional unipotent algebraic group.
     \end{enumerate}
     \end{theorem}

 Secondly, we find that
such groups play a special role in
the the\-ory of modality
 \cite[5.2]{PV94} that goes back to Arnol'd's works on the theory of sin\-gu\-la\-rities \cite{Ar75}.
 The modality of an algebraic action of
$G$ is the maximal number of parameters, on which an algebraic family
 of $G$-orbits of this action may depend. The actions of modality $0$ are precisely that
with a finite number of $G$-orbits.
In this paper we prove that the
groups enjoying the above properties (i), (ii)
can be characterized
in terms of modality.
Namely,
our second main result is the following theorem
(which is  a part of Theorem \ref{mcri} below).

\begin{theorem} \label{t2}
The conditions {\rm (i)}, {\rm (ii)} in
 Theorem {\rm \ref{t1}} hold if and only if
for every $G$-action
$\alpha$ on an irreducible algeb\-raic variety\;$X$,
\begin{equation}\label{mmrr}
\mbox{\it the modality of
$\alpha$ is equal to
$\dim X - \max_{x\in X}\dim G\cdot x$.}
\end{equation}
\end{theorem}

Property  \eqref{mmrr} means that  the maximal number of parameters, on which an algebraic family
 of $G$-orbits in $X$ may depend, is attained on a family, which is open in $X$. If \eqref{mmrr} holds, we say that $\alpha$ is a {\it modality-regular} action.

For a given $G$, Theorems \ref{t1}, \ref{t2} reveal when
the phenomenon
of modality-regu\-la\-r action holds
unconditionally, i.e.,
 for every algebraic action of $G$.
In some
conditional forms the manifestations
 of this phenomenon
   has been disco\-vered
  earlier.

For instance, it follows from \cite[Thms.\;2 and 3]{V86} that if $G$ is
either a Borel subgroup $B$ of a connected complex reductive group $R$
or the unipotent radical $U$ of $B$,
then
the restriction
to $G$ of any action of $R$ on an irreducible algebraic
variety is modality-regular.

Theorems \ref{t1}, \ref{t2} imply that apart from $B$ and $U$, there are other closed subgroups $H$ of $R$ such that the restriction to $H$ of any algebraic action of $R$ on an irreducible algebraic variety is modality-regular. Namely, this property also holds if $H$ is either a torus or a product of a torus and a connected one-dimensional unipotent algebraic group. To the best of my knowledge,
at this writing (April 2020) a complete classi\-fication of subgroups $H$ enjoying this property is not known.

Note that such subgroups $H$ yield examples of non-extendable algebraic actions:
in view of Theorems \ref{t1}, \ref{t2},
if $H$ is not a torus or a product of a torus and a one-dimensional
connected unipotent algebraic group, then there exists
an algebraic action of $H$ on an irreducible algebraic variety $X$, which can
not be extended up to an algebraic action of $R$ on $X$.

Some other conditional forms of the manifestations of the phenomenon of modality-regular action
are found in \cite{Pop17};
for instance, it is proved that if $G$ is reductive, then
every visible action of $G$ on an irreducible affine variety is modality-regular.

  \vskip 1.2mm

This paper is organized as follows.  Our central result, Theorem\;\ref{mcri}, is formulated in Section 4 and proved in Section\;7.\;Section 3 contains the materials about modality, which is necessary
 for
stating Theorem\;\ref{mcri}.
In Sections 5 and 6 are col\-lec\-ted some auxiliary results on property (F) from the statement
of  Theorem\;\ref{mcri} and on modality, which we use in the proof of this theorem.\;In Section\;2
 are collected some conventions, notation, and termi\-no\-logy.

\vskip 2mm

{\it Acknowledgement.} I am grateful  to J.-P.\,Serre for remarks.

\section{\bf 2.\;Conventions, notation, and terminology}
 We fix an algebraically clo\-sed field $k$ of arbitrary characteristic.
 In what follows, we freely use the viewpoint,
standard notation and conventi\-ons of  \cite{B91}, \cite{S98}, \cite{PV94}, where also the proofs of unreferenced claims and/or the rele\-vant references can be found.

 A {\it variety} means an algebraic variety over $k$
(so an algebraic group means
an algebraic group over $k$).
If all irreducible components of a variety $X$ have the same dimension, then
 $X$ is called {\it equidimensional}.
Topological terms are related to the Zariski topology.

 Below all
actions of
algebraic groups
on varieties
are assumed to be
algebraic (i.e., regular/mor\-phic). If an algebraic group $G$ acts on a variety $X$, we say that $X$ is a $G$-{\it variety}.

The product of $d$ copies of
${\mathbf G}_a$ (respectively ${\mathbf G}_m$) is denoted by
${\mathbf G}_a^d$
(respectively ${\mathbf G}_m^d$).
If an algebraic group is isomorphic to ${\mathbf G}_m^d$ for some $d$,
it is called a {\it torus}.

\section{\bf 3.\;Modality}
 Let $H$ be a connected algebraic group.\;Any irreducible $H$-variety $F$ such that all $H$-orbits in $F$ have the same dimension $d$ is called an (algebraic) {\it family} of $H$-orbits depending on
\begin{equation}\label{mo}
{\rm mod}(H:F):=\dim F-d
\end{equation}
\noindent parameters; the integer ${\rm mod}(H:F)$ is called the {\it modality} of $F$.\;If $F\dashrightarrow F\ds H$ is a rational geometric quotient
of this action (i.e., the geometric quotient of an $H$-stable dense open subset of $F$; such a quotient exists by the Rosenlicht  theorem \cite[Thm.\;4.4]{PV94}), then
\begin{equation}\label{mtd}
{\rm mod}(H:F)=\dim F\ds H={\rm tr\,deg}_kk(F)^H
\end{equation}
and $F\ds H$ may be informally viewed as
the variety parametrizing typical $H$-orbits\;in\;$F$.

Given an  $H$-variety $Y$, we denote by ${\mathscr F}(Y)$ the set of all locally closed $H$-stable subsets of $Y$, which are the families. The integer
 \begin{equation}\label{mod}
 {\rm mod}(H:Y):=\underset{F\in {\mathscr F}(Y)}{\rm max}{\rm mod}(H:F),
 \end{equation}
 is then called the {\it modality} of the $H$-variety $Y$.

 If $G$ is a (not necessarily connected) algebraic group and $X$ is a $G$-variety, then by definition,
 \begin{equation*}\label{ccomp}
 {\rm mod}(G:X):={\rm mod}(G^0:X),
 \end{equation*}
where $G^0$ is the identity component of $G$.

It readily follows from the definition that if $Z$ is a locally closed $G$-stable subset of $X$, then
\begin{equation}\label{>>>}
{\rm mod}(G:X)\geqslant {\rm mod}(G:Z).
\end{equation}

 Recall that, for every integer $s$, the set $\{x\in X\mid \dim G\cdot x\leqslant s\}$ is closed in $X$.\;Whence, for every locally closed
 (not necessarily $G$-stable) subset $Z$ in $X$,
 \begin{equation}\label{regreg}
 Z^{\rm reg}:=\{z\in Z\mid \dim G\cd z\geqslant \dim G\cd x\;\mbox{for every $x\in Z$}\}
 \end{equation}
is a nonempty
open subset of $Z$.

 The definition of modality implies that \eqref{mod} still holds
 if ${\mathscr F}(Y)$ is re\-placed by the set of all  maximal (with respect to inclusion) families in $Y$, i.e., by the {\it sheets} of $Y$ \cite[Sect.\;6.10]{PV94}.\;Recall
 that there are only finitely many sheets of $Y$. If $Y$ is irreducible,
 then
 $Y^{\rm reg}$ is a sheet, called {\it regular}, which is open and dense in\;$Y$.
 By \eqref{mtd},
 \begin{equation}\label{mtr}
 {\rm mod}(H:Y^{\rm reg})={\rm tr}\,{\rm deg}_k k(Y)^H.
 \end{equation}
Similarly,
\eqref{mod} still holds if  ${\mathscr F}(Y)$ is replaced by the set of all
$H$-stable irredu\-cible locally closed (or closed) subsets of $Y$, and ${\rm mod}(H:F)$ by
${\rm tr\,deg}_kk(F)^H$.

The aforesaid shows that ${\rm mod}(G:X)\!=\!0$ if and only if the set of all $G$-or\-bits in $X$ is finite.

  The existence of regular sheets leads to defining the following distingui\-shed class of algebraic group actions:

  \begin{definition}\label{reg} An irreducible $G$-variety $X$ and the action of $G$ on $X$
  are called {\it modality-regular} if ${\rm mod}(G:X)\!=\!{\rm mod}(G:X^{\rm reg})$.
 \end{definition}

It follows from \eqref{mo}, \eqref{regreg}, and Definition \ref{reg} that an irreducible $G$-variety $X$ is modality-regular if and only if
\begin{equation*}
{\rm mod}(G:X)=\dim X-\max_{x\in X}\dim G\cd x.
\end{equation*}

\section{\bf 4.\;Main result: formulation}

 \begin{theorem}\label{mcri}
  Let $G$ be a connected affine algebraic group.
 The following proper\-ties
 are equivalent:
 \begin{enumerate}[\hskip 2.2mm\rm(i)]
 \item[$({\rm F})$]
 there are only finitely many $\,G$-orbits in every $G$-va\-riety
 contain\-ing
an open $G$-orbit;
  \item[$({\rm M})$]
  every irreducible $G$-variety
  is  modality-regular;
 \item[$({\rm G})$] 
 $G$ is
 either a torus or a direct product of a torus and
a connected one-dimensional unipotent algebraic group.
 \end{enumerate}
 \end{theorem}

 \begin{remarks}\

1. Recall  that every connected one-dimensional unipo\-tent algebraic group is isomorphic to $\mathbf G_a$ (see, e.g., \cite[3.4.9]{S98}).

2.  Let $G$ be a connected affine algebraic group. The following pro\-perties are
 equiva\-lent:
 \begin{enumerate}[\hskip 2mm\rm (i)]
 \item $G$ is a product of a torus and
a connected one-dimensional uni\-po\-tent algebraic group.
 \item $G$ is nilpotent and its
 unipotent radical
 is one-dimensional.
 \end{enumerate}
Therefore, property (G) is equivalent to the property
 \begin{enumerate}[\hskip 4.2mm\rm (i)]
\item[$({\rm G}')$]  $G$ is nilpotent and its unipotent radical is at most one-dimen\-sional.
\end{enumerate}
 \end{remarks}

\section{\bf 5.\;Auxiliary results: property (F)}
 This section contains some auxi\-liary results about property (F) from the
formulation
 of Theorem \ref{mcri} that will be used in its proof.
First we explore  behaviour of this property
under passing to a subgroup and a quotient group. Then we explore it for two-dimensional
connected solv\-able affine algebraic groups, and, in conclusion, for
semisimple affine algebraic groups.

\subsubsection{\bf \boldmath$5.1.$\;Passing to a subgroup and a quotient group}

 \begin{lemma}\label{reduction} Let $G$ be a connected affine algebraic group and let $H$ be its closed
 subgroup.\;If
 $\,G$ enjoys property ${\rm (F)}$, then
 \begin{enumerate}[\hskip 4.2mm \rm(a)]
 \item $H$ enjoys property ${\rm (F)}$;
 \item $G/H$, for normal $H$, enjoys property ${\rm (F)}$.
 \end{enumerate}
 \end{lemma}

 \begin{proof}
 (a) Arguing on the contrary, suppose there exists  an irreducible $H$-variety $Y$ with infinitely many $H$-orbits, one of which, say, $\mathcal O$,
  is open in $Y$.\;Since the action canonically lifts to the normalization \cite{Ses63},
we may (and shall) assume that $Y$ is normal.\;Then, by \cite[Lem.\;8]{S74}, we have
 $Y=\bigcup_{i\in I} U_i$, where each $U_i$ is an $H$-stable quasi-projective open subset of $Y$.\;As $Y$ is irreducible, each $U_i$ contains $\mathcal O$.\;Since in the Zarisky topology any open covering contains a finite subcovering,
there is $i_0\in I$ such that $U_{i_0}$
 contains infinitely many $H$-orbits.\;There\-fo\-re
 replacing $Y$ by $U_{i_0}$, we may (and shall) assume that
 $Y$ is quasi-projective.\;Then, by  \cite[3.2]{Ser58} (see also \cite[Thm.\;4.9]{PV94}),
 the homogeneous fiber space $X:=G\times^H Y$ over $G/H$ with the fiber $Y$ is an algebraic variety.\;Since for the action of $H$ on $Y$ there are infinitely many orbits one of which is open,
 the natural action of $G$ on $X$ enjoys these properties as well; see \cite[Thm.\;4.9]{PV94}. This contradicts the condition that $G$ enjoys property (F), thereby proving (a).

(b) Assume,  again arguing on the contrary, that there is an irreducible algebraic $G/H$-variety $X$ with infi\-ni\-tely many $G/H$-orbits, one of which is open in $X$.\;Since the canonical homomorphism $G\to G/H$ determines an action of $G$ on $X$ whose orbits coincide with $G/H$-orbits, this contra\-dicts the condition that $G$ enjoys property (F), thereby pro\-v\-ing\;(b).
 \end{proof}

\subsubsection{\bf \boldmath$5.2.$ \hskip -1mmTwo-dimensional connected solvable  affine algebraic groups}
Let $S$ be a two-dimensional
connected solvable  affine algebraic group. Then
$S=T\ltimes S_u$, where $T$ is  a maximal torus and $S_u$ is the unipotent radical of $S$.\;There are only the following three possibilities (S1), (S2), and (S3) for $S$:

\vskip 1mm

\begin{asparaitem}[\hskip 0mm$\bullet$ ]
\item[$\rm (S1)$] {\it $S_u$ is trivial,
i.e., $S$ is a torus.}

Such $S$ enjoys pro\-per\-ty (F) (see the first example in Introduction).

\vskip 1mm

    \item[$\rm (S2)$] {\it The following equality holds: $\dim T=\dim S_u=1$.}

         In this case, $T$ and $S_u$ are isomorphic to respectively to ${\mathbf G}_m$ and ${\mathbf G}_a$ and
         there is $d\in \mathbf Z$ such that $S$ is isomorphic to the group $S(d):={\mathbf G}_m \ltimes{\mathbf G}_a$, in which the group operation is defined by the formula
        \begin{equation}\label{S(n)}
(t_1, u_1)(t_2, u_2):=(t_1t_2, t_2^du_1+u_2), \quad t_i\in \mathbf G_m, u_i\in \mathbf G_a.
        \end{equation}

 Indeed,
  fix an isomorphism $\theta\colon {\mathbf G}_a\to S_u$.\;For any $t\in T$, the map $S_u\to S_u$, $u\mapsto tut^{-1}$,  is an automorphism, therefore, there is a character
 $\chi\colon T\to {\mathbf  G}_m$ such that      $t\theta(u)t^{-1}=\theta(\chi(t)u)$ for all $u\in {\mathbf G}_a$, $t\in T$; whence the claim.

 The group $S(d)$ is commutative if and only if $d=0$.

 \begin{proposition}\label{S(n)neq0}
  Every group $S(d)$ for $d\neq 0$ does not enjoy property ${\rm(F)}$.
  \end{proposition}
  \begin{proof} It follows from \eqref{S(n)} that
  \begin{equation*}
  S(d)\to {\rm GL}_2,\quad (t, u)\mapsto
  \begin{pmatrix}
  t^d& 0\\ u&1
  \end{pmatrix},
  \end{equation*}
  is a representation of $S(d)$.\;It determines the following linear action of $S(d)$ on ${\mathbf A}^{\hskip -.5mm 2}$:
  \begin{equation}\label{action}
  g\cdot a:=(a_1t^d, a_1u+a_2)\;\;\mbox{for $g=(t, u)\in S(d)$,  $a=(a_1, a_2)\in {\mathbf A}^{\hskip -.5mm 2}$.}
  \end{equation}

 From \eqref{action}
  and $d\neq 0$ we immediately infer that the fixed point set of this action is the line $\ell:=\{(a_1, a_2)\in {\mathbf A}^{\hskip -.5mm 2}\mid a_1=0\}$ whose complement ${\mathbf A}^{\hskip -.5mm 2}\setminus \ell$ is a single orbit.\;Thus, the $S(d)$-variety ${\mathbf A}^{\hskip -.5mm 2}$ contains infinitely many orbits one of which is open. This completes the proof.
 \end{proof}

 \item[$\rm (S3)$] {\it $T$ is trivial,
 i.e., $S$ is unipotent.}
\end{asparaitem}

 \begin{proposition}\label{2u} Every  two-dimensional connected unipotent affine
 al\-geb\-raic group $S$ does not enjoy property ${\rm (F)}$.
 \end{proposition}

\begin{proof} First, by \cite[6.3.4]{S98} there exists a one-dimensional connected closed subgroup of $S$ lying in the center of $S$. Since $\dim S=2$, this implies that
there is an exact sequence of group homomorphisms
\begin{equation}\label{exseq}
\begin{gathered}
0\rightarrow\mathbf G_a\xrightarrow{\iota} S\rightarrow\mathbf G_a\rightarrow 0,
\end{gathered}
\end{equation}
such that $\iota(\mathbf G_a)$ is a central subgroup of $S$. Thus $S$ is a central
extension of ${\mathbf G_a}$ by ${\mathbf G_a}$.

We consider two cases,  of which the first is simpler
than the second.

 The consideration  in the first case is based on
the following

\vskip 2mm

\noindent{{\bf Claim.} \it The group ${\mathbf G}_a^d$
does not enjoy property ${\rm(F)}$ for every $d\geqslant 2$.}

\begin{proof}[Proof of Claim]
 The action
 of ${\mathbf G}_a^d$ on itself by translations is the action
 on the
 affine space ${\mathbf A}^{\hskip -.5mm d}$ defined
 by the formula
 \begin{gather*}
 u\cdot a
 :=(a_1+u_1,\ldots, a_d+u_d)\\[-1mm]
 \mbox{for $u=(u_1,\ldots, u_d)\in {\mathbf G}_a^d$,\; $a=(a_1,\ldots, a_d)\in {\mathbf A}^{\hskip -.5mm d}$.}
 \end{gather*}

  We indentify ${\mathbf A}^{\hskip -.5mm d}$ with the affine chart
 \begin{equation}\label{chart}
 {\mathbf P}^d_d:=\{(b_0:\ldots:b_d)\in {\mathbf P}^d\mid b_d\neq 0\}
  \end{equation}
  of the projective space ${\mathbf P}^d$. Then the following formula extends this action
 to the action of  ${\mathbf G}_a^d$
 on ${\mathbf P}^d$:
 \begin{gather*}
 u\cdot b
 :=(b_0+u_1b_d:\ldots: b_{d-1}+u_db_d:b_d)\\[-1mm]
 \mbox{for $u=(u_1,\ldots, u_d)\in {\mathbf G}_a^d$, $b=(b_0:\ldots:b_d)
 \in {\mathbf P}^d$}.
 \end{gather*}
For the latter action, ${\mathbf P}^d_d$
is an open orbit, and the hyperplane ${\mathbf P}^d\setminus
{\mathbf P}^d_d$
is pointwise fixed.\;For $d\geqslant 2$, this hyperplane contains infinitely many points, whence 
the claim.
 \end{proof}

The two cases mentioned above
are that of ${\rm char}\,k=0$  and ${\rm char}\,k\neq 0$.
Note that in the first version of this paper \cite{Pop17_1} only the simpler case of ${\rm char}\,k=0$ has been considered,  as at the time of writing we did not have a proof for ${\rm char}\,k\neq 0$.

\vskip 2mm

\noindent{\it First case: ${\rm char}\,k=0$}.

\vskip 1mm

The fact that in this case $S$ does not enjoy property (F) immediately follows from the above Claim, because ${\rm char}\,k=0$ implies that
$S$ is isomorphic to $\mathbf G_a^2$.
Indeed, using the notation of \eqref{exseq}, take an element in $S\setminus \iota(\mathbf G_a)$. Since ${\rm char}\,k=0$, 
it lies in a
connected closed one-dimensional subgroup $D$ of $S$; see \cite[Rem. in II.7.3]{B91}.
Both $\iota(\mathbf G_a)$ and $D$ are isomorphic to $\mathbf G_a$,
centralize each other, and $\iota(\mathbf G_a)\cap D=e$
because, in view of ${\rm char}\,k=0$,
$\mathbf G_a$ is the closure of any its nontrivial cyclic subgroup.
This implies that $\iota(\mathbf G_a)\times D\to S$, $(g, h)\mapsto gh$ is an embedding of algebraic groups, and therefore, an isomorphism because $\dim S=2$.

\vskip 2mm

\noindent{\it Second case: ${\rm char}\,k=p>0$}.

\vskip 1mm

In this case, it is no longer true that every two-dimensional connected unipotent affine algebraic group (even commutative) is isomorphic to $\mathbf G_a^2$. In \cite{Pop19}, we proved that the two-dimensional Witt group $W_2(p)$ does not enjoy property (F).
Below we give a proof in the general case, developing the idea
used in
 \cite{Pop19}.

Since the underlying variety of $S$ is isomorphic to
$\mathbf A^{\hskip -.5mm 2}$ (see \cite[Prop.\;2]{G58}), we can (and shall)
identify $S$ with  $\mathbf A^{\hskip -.5mm 2}$
endowed with a group operation $\ast$, which is described below.

Let $x_0, x_1\in k[\mathbf A^{\hskip -.5mm 2}]$ be the standard coordinate functions on\;$\mathbf A^{\hskip -.5mm 2}$:
\begin{equation}\label{coo}
    x_i(a)=a_i\quad \mbox{for every $a=(a_0, a_1)\in \mathbf A^{\hskip -.5mm 2}$.}
\end{equation}

The classes of equivalent central extensions of $\mathbf G_a$ by $\mathbf G_a$
are classified by
$H^2(\mathbf G_a, \mathbf G_a)$ (see, e.g., \cite[Chap.\,2]{DFPS08}).
In
\cite[II, \S3, 4.6]{DG70}, it is shown that $H^2(\mathbf G_a, \mathbf G_a)$ is a free
${\rm End}(\mathbf G_a)$-module, having the following family of
$2$-cocycles as a basis modulo the coboundaries:
\begin{equation}\label{basi}
\begin{split}
&\sum_{i=1}^{p-1}c_ix_0^ix_1^{p-i}, \quad
\mbox{where}
\;\; c_i:=
{(p-1)!}/{i!(p-i)!}, \\[-1mm]
&x_0x_1^{p^j},\quad j=1, 2,\ldots,
\end{split}
\end{equation}
(informally, the first polynomial in \eqref{basi}
is $
\big((x_0^p+x_1^p)-(x_0+x_1)^p\big)/p$.) In view of the existence of an exact
sequence \eqref{exseq}, this implies that  there are polynomials $f, h_1,\ldots, h_d\in k[z]$ (where $z$ is a variable over $k$) and positive integers $m_1,\ldots, m_d$ such that,
for all $(a_0, a_1), (b_0, b_1)\in \mathbf A^{\hskip -.5mm 2}$, we have
\begin{equation}\label{G*}
\begin{split}
(a_0, a_1)&* (b_0, b_1)\\[-3mm]
&=\Big(a_0+b_0, a_1+b_1 +f\Big(\sum_{i=1}^{p-1}c_ia_0^ib_0^{p-i}\Big)+
\sum_{j=1}^{d}h_j\big(a_0b_0^{p^{m_j}}\big)\!\Big).
\end{split}
\end{equation}

We consider $k[x_0, x_1]=k[S]$ as an $S$-module with respect to the action of $S$ by left translations (see \cite[I.1.9]{B91}).
It then follows from \eqref{G*}, \eqref{coo}  that for every $s=(s_0, s_1)\in S$,
we have
\begin{align}
s^{-1}\cdot x_0&=x_0+s_0,\label{ltf1}\\[-2.8mm]
s^{-1}\cdot x_1&=x_1+s_1+f\Big(\sum_{i=1}^{p-1}c_is_0^ix_0^{p-i}\Big)+
\sum_{j=1}^{d}h_j\big(s_0x_0^{p^{m_j}}\big).\label{ltf2}
\end{align}

We now fix an integer $n$ such that
\begin{equation}\label{boundn}
{\rm deg}\Big(f\Big(\sum_{i=1}^{p-1}c_ix_0^ix_1^{p-i}\Big)+
\sum_{j=1}^{m}h_j\big(x_0x_1^{p^{m_j}}\big)\!\Big)\leqslant n.
\end{equation}

It then follows from \eqref{ltf1}, \eqref{ltf2}, and \eqref{boundn}
that the $k$-linear span $V$ of the sequence of functions
\begin{equation}\label{basiss}
1,\; x_0,\; x_0^2,\;\ldots,\; x_0^{n},\; x_1
\end{equation}
is an $(n+2)$-dimensional $S$-stable linear subspace of $k[x_0, x_1]$, and\;if
$$\rho_V\colon S\to {\rm GL}(V)$$
is the linear representation determined by the natural action of $S$ on $V$, then
there are polynomials $q_0, q_1,\ldots, q_{n-1}\in k[z]$ such that
\begin{equation}\label{degq}
{\rm deg}(q_i)\leqslant n\quad\mbox{for all $i$},
\end{equation}
and, for every $s=(s_0, s_1)\in S$,
the
matrix
of
$\rho_V(s^{-1})$ relative to basis \eqref{basiss} has the form

\begin{equation}\label{rhoVg}
\arraycolsep=3.9pt
R_{g^{-1}}:=\begin{pmatrix}
1&s_0&s_0^2&s_0^3&s_0^4&\ldots&s_0^{n}&s_1\\
0& 1& \tbinom{2}{1}
s_0 & \tbinom{3}{2}
s_0^2& \tbinom{4}{3}
s_0^3 &\ldots &\tbinom{n}{n-1}
s_0^{n-1}& q_{n-1}(s_0)\\
0& 0 & 1 & \tbinom{3}{1}
s_0  &\tbinom{4}{2}
s_0^2 &\ldots & \tbinom{n}{n-2}
s_0^{n-2}&q_{n-2}(s_0)\\
0& 0& 0& 1& \tbinom{4}{1}
s_0 & \ldots &
\tbinom{n}{n-3}
s_0^{n-3}& q_{n-3}(s_0)\\
\vdots &\vdots &\vdots &\vdots &\vdots &\ddots &\vdots &\vdots\\
0&0&0&0&0&1& \tbinom{n}{1}
s_0&q_1(s_0)\\
0&0&0&0&0&0&1&q_0(s_0)\\
0&0&0&0&0&0&0&1
\end{pmatrix}.
\end{equation}

\vskip 2mm

Next, we fix a positive integer $r$.
It follows from \eqref{ltf1} and ${\rm char} \,k=p$
that 
the $k$-linear span $U$ of the sequence of functions
\begin{equation}\label{2space}
1,\; x_0^{p^r}
\end{equation}
is a two-dimensional $S$-stable linear subspace of $k[x_0, x_1]$, and if $\rho_U\colon S\to {\rm GL}(U)$ is the linear representation determined by the natural action of $S$ on $U$, then, for every $s=(s_0, s_1)\in S$, the matrix of $\rho_U(s^{-1})$
relative to basis \eqref{2space} has the form
\begin{equation}\label{rhoUg}
P_{s^{-1}}:=\begin{pmatrix}
1&s_0^{p^r}\\
0&1
\end{pmatrix}.
\end{equation}

We now identify $V\oplus U$ with the affine space
${\mathbf A}^{n+4}$ by means of the bijection
\begin{gather*}
V\oplus U\to {\mathbf A}^{n+4},\\[-3.5mm]
\Big(\displaystyle\sum_{i=0}^{n}a_ix_0^i
+a_{n+1}x_1, a_{n+2}+a_{n+3}x_0^{p^r}\Big)
\mapsto  (a_0, a_1,\ldots, a_{n+3}),
\end{gather*}
and, in turn, identify ${\mathbf A}^{n+4}$ with the standard affine chart
\begin{equation*}
{\bf P}^{n+4}_{n+4}:=\{(b_0:b_1:\ldots :b_{n+3})
\mid
b_{n+3}\neq 0\}
\end{equation*}
of the projective space ${\bf P}^{n+4}$ by means of the bijection
\begin{equation*}
{\mathbf A}^{n+4}\to {\bf P}^{n+4}_{n+4},\quad
(a_0, a_1,\ldots, a_{n+3})\mapsto (a_0: a_1:\ldots: a_{n+3}:1).
\end{equation*}

The linear action of $S$ on ${\bf A}^{n+4}$ determined by
$\rho_V\oplus \rho_U$
 extends to the following action of $S$ on
${\bf P}^{n+4}$: if
$s\in S$ and
\begin{equation}\label{WrhoQ}
 Q_{s^{-1}}:=\begin{pmatrix}
 R_{s^{-1}}& 0\\
 0& P_{s^{-1}}
 \end{pmatrix}\!,
\end{equation}
then
\begin{equation}\label{aPP}
\begin{gathered}
s^{-1}\cdot (b_0:b_1:\ldots :b_{n+4}):=(b'_0:b'_1:\ldots :b'_{n+4}),\;
\mbox{where}\\[-.7mm]
(b'_0:b'_1:\ldots :b'_{n+3})^{\sf T}=Q_{s^{-1}}(b_0:b_1:\ldots :b_{n+3})^{\sf T},\;\; b'_{n+4}=b_{n+4}.
\end{gathered}
\end{equation}

It follows from \eqref{rhoVg}, \eqref{rhoUg}, \eqref{WrhoQ},
\eqref{aPP} that the
line
\begin{equation}\label{ell}
    \ell:=\big\{(b_0:b_1:\ldots :b_{n+4})\in {\bf P}^{n+4}\mid
    \mbox{$b_i=0$ for every $i\neq 0, n+2$}\big\}
\end{equation}
is pointwise fixed with respect to this action of $S$ on
${\bf P}^{n+4}$:
\begin{equation}\label{fff}
\ell\subseteq ({\bf P}^{n+4})^{S}.
\end{equation}

We shall now show that if we take $r$ such that
\begin{equation}\label{r>}
p^r>n,
\end{equation}
then  there exists an $S$-orbit in ${\mathbf P}^{n+4}$, whose closure contains the line $\ell$;
this and \eqref{fff} will then complete the proof.

Namely, assume that \eqref{r>} holds
and consider the $S$-orbit of the point
\begin{equation*}
    v:=(1:1:\ldots: 1)\in {\bf P}^{n+4}.
\end{equation*}
It follows from \eqref{rhoVg}, \eqref{rhoUg}, \eqref{WrhoQ},
\eqref{aPP} that this orbit
$S\cdot v$
is the image of the morphism
\begin{equation}\label{varphi}
\begin{gathered}
\varphi\colon {\bf A}^2\to {\bf P}^{n+4},\quad
s\mapsto
(f_0(s
)\!:\!f_1(s
)\!:\!\ldots \!:\!f_{n+3}(s
)\!:\!1),
\end{gathered}
\end{equation}
where
\begin{equation}\label{Worbit}
\left.
\begin{split}
f_0
&=1+
x_0+
x_0^2+
\cdots+
x_0^{n}+
x_1,\\
f_1
&=1+\tbinom{2}{1}
x_0+\tbinom{3}{2}
x_0^2+
\cdots+\tbinom{n}{n-1}
x_0^{n-1}+q_{n-1}(
x_0),\\
f_2
&=1+\tbinom{3}{1}
x_0+\tbinom{4}{2}
x_0^2+\cdots+\tbinom{n}{n-2}
x_0^{n-2}+
q_{n-2}(
x_0),\\
f_3
&=1+\tbinom{4}{1}
x_0+\tbinom{5}{2}
x_0^2+\cdots+\tbinom{n}{n-3}
x_0^{n-3}
+q_{n-3}(
x_0),\\
\;.\;.&\;.\;.\;.\;.\;.\;.\;.\;.\;.\;.\;.\;.\;.\;.\;.\;.\;.\;.\;.\;.\;.\;.\;.\;.\;.\;.\;.\;.\;.\;.\;.\;.\;.\;.\;.\;.\;.\;.\;.\;\\
f_{n-1}
&=1+\tbinom{n}{1}
x_0+q_1(
x_0),\\
f_{n}
&=1+q_0(
x_0),\\
f_{n+1}
&=1,\\
f_{n+2}
&=1+
x_0^{p^r},\\
f_{n+3}
&=1.
\end{split}
\right\}
\end{equation}

Let ${\bf A}^1_*:={\bf A}^1\setminus \{0\}$. We assign to every $a\in k$
the morphism
\begin{equation}\label{psia}
    \psi_a\colon {\bf A}^1_*\to {\bf A}^2,\;\;t\mapsto (t^{-1}, at^{-p^r}).
\end{equation}
Formulas \eqref{varphi}, \eqref{Worbit}, \eqref{psia} show that $(\varphi\circ \psi_a)(t)$
for every $t\in {\bf A}^1_*$ is the point
\begin{equation*}\label{psiphi}
\begin{split}
   \big(f_0(t^{-1}\!, at^{-p^r})\!:\!f_1(t^{-1}\!, at^{-p^r})\!:\!\ldots \!:\!
   f_{n}(t^{-1}\!, at^{-p^r})\!:\!1\!:\!1+t^{-p^r}\!\!:\!1\!:\!1\big)\in {\mathbf P^{n+4}}.
   \end{split}
\end{equation*}
Since
$t\neq 0$, this point
coincides with
\begin{equation}\label{psiphicoo}
\big(t^{p^r}\hskip -1.2mm f_0(t^{-1}, at^{-p^r}\!)\!:\!t^{p^r}\hskip -1.2mmf_1(t^{-1}, at^{-p^r}\!)\!:\!\ldots \!:\!t^{p^r}\hskip -1.2mm f_{n}(t^{-1}, at^{-p^r}\!)
\!:\!
1+t^{p^r}\hskip -1.2mm\!:\!t^{p^r}\hskip -1.2mm\!:\!t^{p^r}\big).
\end{equation}

Denote by $k[z]_+$ the $k$-linear span of $\{z^d \mid d>0\}$ in $k[z]$. From \eqref{Worbit}, \eqref{r>}, and \eqref{degq} we infer the existence of
polynomials $u_0, u_1, \ldots, u_n\in k[z]_+$ such that
\begin{equation}\label{coocase}
t^{p^r}f_i(t^{-1}, at^{-p^r})=\begin{cases}
a+u_0(t)&\mbox{for $i=0$},\\
u_i(t)&\mbox{for $i=1,\ldots, n$.}
\end{cases}
\end{equation}

From \eqref{psiphicoo} and \eqref{coocase} we conclude that the morphism
\begin{equation*}
    \varphi\circ\psi_a\colon {\bf A}^1_*\to S\cdot v\subseteq {\mathbf P}^{n+4}
\end{equation*}
uniquely extends up to a morphism
$$
\varepsilon_a\colon {\bf A}^1\to {\bf P}^{p+4}
$$
such that
$$
\varepsilon_a(0)=(a:0:0:\ldots:0:1:0:0).
$$

By \eqref{ell} we have $\varepsilon_a(0)\in \ell$. Since $a$ is an arbitrary element of $k$, we infer that the closure of $S\cdot v$ in $\mathbf P^{n+4}$ contains a nonempty subset of $\ell$; hence it contains the whole $\ell$ as well.
This completes the proof of  Pro\-po\-si\-tion\;\ref{2u}.
\end{proof}

\subsubsection{\bf \boldmath
$5.3.$\;Semisimple  affine algebraic groups}

  \begin{proposition}\label{ss} Every nontrivial connected semisimple algebraic group $\,G$ does not enjoy property ${\rm(F)}$.
 \end{proposition}

 \begin{proof} Let $\alpha$ be a root of $G$ with respect to a maximal torus and let  $G_\alpha$ be the centralizer of the torus $({\rm ker}\,\alpha)^0$ in $G$.\;The commutator group
 $(G_\alpha, G_\alpha)$ is isomorphic to either ${\rm SL}_2$ or ${\rm PSL}_2$ (see, e.g., \cite[7.1.2, 8.1.4]{S98}).\;Correspon\-dingly, the Borel subgroups of $(G_\alpha, G_\alpha)$ are isomorphic to either $S(1)$ or
 $S(2)$ (see Subsection 5.2(S2) above).\;Hence, by Pro\-po\-sition \ref{S(n)neq0}, they do not enjoy  pro\-perty ${\rm(F)}$.\;The claim then follows from Lemma \ref{reduction}(a).
 \end{proof}

\section{\bf 6.\;Auxiliary results: modality} The following lemma helps to prac\-ti\-cally compute
the modality and will be used in
the proof of Theorem\;\ref{mcri}.

\begin{lemma}\label{SI} Let $G$ be an algebraic group, let $X$
be a $G$-variety,
and let
 $\{C_i\}_{i\in I}$ be a collection of  the subsets of $X$ such that
\begin{enumerate}[\hskip 2.0mm\rm(i)]
\item $I$ is finite;
\item $\bigcup_{i\in I}C_i=X$;
\item the closure $\overline{C_i}$ of $\,C_i$ in $X$ is irreducible for every $i\in I$;
\item every $C_i$ is $G$-stable;
\item all $G$-orbits in $C_i$ have the same dimension $d_i$ for every $i\in I$.
\end{enumerate}
Then the following hold:
\begin{enumerate}[\hskip 4.2mm\rm(a)]
\item ${\rm mod}(G:X)=
{\rm max}_{i\in I}\big(\!\dim \overline{C_i}-d_i\big)$;
    \item if $X$ is irreducible, then $X\!=\!\overline{C_{i_0}}$ for some $i_0$,
    and ${\rm mod}(G:X^{\rm reg})\!=\!\dim X\!-\!d_{i_0}$.
\end{enumerate}
\end{lemma}

\begin{proof}
By (iii), we have a family ${\overline {C_i}}^{\rm reg}$, and (v) implies
$C_i\subseteq {\overline {C_i}}^{\rm reg}$.
Whence
\begin{equation}\label{Cm}
{\rm mod}(G:\overline {C_i}^{\rm reg})=\dim \overline {C_i}-d_i.
\end{equation}
From \eqref{>>>}
and \eqref{Cm}, we infer that
${\rm mod}(G:X)\geqslant \underset{i\in I}{\max} (\dim {\overline C}_i-d_i)$.

To prove the op\-po\-site inequality let $Z\in {\mathscr F}(X)$ be a family of $s$-dimensional $G$-orbits such that ${\rm mod}(G:X)=\dim Z-s$
and let
$J:=\{i\in I \mid Z\cap C_i\neq \varnothing\}.$
 By (ii), we have $Z=\bigcup_{j\in J}(Z\cap {\overline {C_j}})$.\;Since $Z$ is irreducible and,
  by (i),  $J$ is finite, there is $j_0\in J$ such
 that $Z\subseteq {\overline {C_{j_0}}}$.\;As $Z\cap C_{j_0}\neq \varnothing$,
 we have $s=d_{j_0}$.\;Therefore,
 ${\rm mod}(G:X)=\dim Z-s\leqslant \dim
 \overline {C_{j_0}}-d_{j_0}$.\;This
 proves (a).

By (ii), $\bigcup_{i\in I} \overline{C_i}=X$.\;If $X$ is irreducible, then, in view of (i), this equality implies the existence of $i_0$ such that $X=\overline{C_{i_0}}$.\;This and \eqref{Cm}
prove (b).
\end{proof}

\begin{lemma}\label{moma} Let $G$ be a connected algebraic group, let $X$ and $Y$ be the irreducible $G$-varieties,  and let
$\varphi\colon X\dashrightarrow Y$ be a  rational $G$-equivariant
map.
\begin{enumerate}[\hskip 2mm\rm(i)]
\item If $\varphi$ is dominant, then
${\rm mod}(G:X^{\rm reg})\geqslant {\rm mod}(G:Y^{\rm reg}).$\;If, more\-over,
$\dim X=\dim Y$, then ${\rm mod}(G:X^{\rm reg})= {\rm mod}(G:Y^{\rm reg}).$
\item If $\varphi$  is a surjective morphism, then
${\rm mod}(G:X)\geqslant {\rm mod}(G:Y).$
\end{enumerate}
\end{lemma}

\begin{proof} The inequality in (i) follows from \eqref{mtr} because $\varphi$  determines
a $G$-equi\-variant field embedding  $\varphi^*\colon k(Y)\hookrightarrow k(X)$.

Assume that
$\dim X=\dim Y$. Then, by the fiber dimension theorem, the fibers of $\varphi$ over the points of an open subset of $Y$ are finite.\;Whence, for every point $x$ of an open subset of $X$, the equality
$\dim \,G\cd x=\dim \,G\cd \varphi(x)$ holds. This implies the following equality:
\begin{equation}\label{mXmY}
m_X:=\max_{x\in X}\dim G\cdot x=m_Y:= \max_{y\in Y}\dim G\cdot y.
\end{equation}
From \eqref{mXmY} and \eqref{mo} we obtain
 ${\rm mod}(G:X^{\rm reg})=\dim X-m_X=\dim Y-m_Y={\rm mod}(G:Y^{\rm reg}).$
This proves (i).

To prove (ii), consider a family $F$ in $Y$ such that
\begin{equation}\label{FY}
{\rm mod}(G:Y)={\rm mod}(G:F).
\end{equation}

If $\varphi$ is a surjective morphism, then $\varphi\colon \varphi^{-1}(F)\to F$ is a surjective
morphism. As $F$ is irreducible, there is an irreducible component $\widetilde F$ of $\varphi^{-1}(F)$ such that $\varphi\colon \widetilde F\to F$ is a surjective morphism.\;Since $\varphi$ is $G$-equivariant and $G$ is connected, $\widetilde F$ is $G$-stable, so the latter morphism is $G$-equivariant.\;Hence
\begin{equation}
\label{>=}
\begin{split}
{\rm mod}(G:X)
&\geqslant {\rm mod}(G:{\widetilde F}^{\rm reg})
\geqslant
{\rm mod}(G:F^{\rm reg})\\
&={\rm mod}(G:F)
={\rm mod}(G:Y)
\end{split}
\end{equation}
(in \eqref{>=}, the first inequality follows from \eqref{mod}, and the second from (i); the first equality follows from $F=F^{\rm reg}$,
and the second from \eqref{FY})
This proves (ii).
\end{proof}

Recall \cite{Ses63} that
if $G$ is an algebraic group, $X$ is an irreducible $G$-variety, and $\nu\colon \widetilde X
\to X$ is its normalization, then the action of $G$ on $X$
lifts to $\widetilde X
$ so that
$\nu$ becomes $G$-equivariant.

\begin{corollary}\label{normal} In the above notation,

\begin{enumerate}[\hskip 2mm \rm (i)]
\item
${\rm mod}(G: \widetilde X)
={\rm mod}(G:X);$
\item 
$X$ is modality-regular if
and only if
$\widetilde X
$\;is.
\end{enumerate}
\end{corollary}

  \begin{lemma}\label{torus} Let $T$ be a torus and let $Y$ be an irreducible $T$-variety. Then
  \begin{enumerate}[\hskip 2mm \rm (i)]
  \item the stabilizer of any point of an open subset of $\,Y$\;is
the kernel of $T$-action on $Y$;
      \item 
$Y$ is modality-regular.
      \end{enumerate}
  \end{lemma}
  \begin{proof}
First, we may (and shall) assume that $T$ acts on $Y$ faithfully.\;Next, by Corollary \ref{normal}, replacing $Y$ by $\widetilde Y$, we may (and shall) assume that $Y$ is normal.\;By \cite[Cor.\;2, p.\;8]{S74}, then $Y$ is covered by $T$-stable
  affine open subsets.\;Whence
   we may (and shall) assume that $Y$ is affine.

  (i) As $Y$ is affine,  by \cite[Thm.\;1.5]{PV94}, we may (and shall) assume that $Y$ is a closed $T$-stable subset of a finite-dimensional algebraic $T$-module and $Y$ does not lie in a proper $T$-submodule of $V$.
The action of $G$ on $V$ is faithful because that on $Y$ is.\;As $T$ is a torus, $V$ is a direct sum of the $T$-weight subspaces.\;Let $U$ be the complement in $V$ to the union of these subspaces.\;The stabilizer of any point of $U$ coincides with the kernel of the action on $V$,
  hence is trivial. As, by construction, $Y\cap U\neq \varnothing$, this proves (i).

  (ii) In view of
  Definition\;\ref{reg}, the proof  of \cite[Prop.\;1]{V86} can be viewed  as
   that of (ii).\;For the sake of completeness, below is a different proof.

By (i), we
have
   $ {\rm mod}(T:Y^{\rm reg})=\dim Y-\dim T. $
Let $S$ be a sheet in the $T$-variety $Y$, and let $T_0$ be the kernel of the action of $T$ on $S$.\;Then
  $ {\rm mod}(T:S)=\dim S-(\dim T-\dim T_0).$

  Since $Y$ is affine, the morphism $\pi\colon Y\to Y/\!\!/T_0=:
{\rm Spec}\,k[Y]^{T_0}$ induced by the identity embedding  $k[Y]^{T_0}\hookrightarrow k[Y]$
is the categorical quotient for the action of $T_0$ on $Y$.\;As $T_0$ act on $Y$ faithfully, (i) yields
  $ \dim Y/\!\!/T_0\leqslant \dim Y-\dim T_0.$

  Since $S$ is pointwise fixed by $T_0$ and $\pi$ separates closed $T_0$-orbits
(see \cite[Thm.\;9.4]{PV94}),
  we have $\dim \pi (S)=\dim S$; whence
  $ \dim S\leqslant \dim Y/\!\!/T_0
  \leqslant  \dim Y-\dim T_0.$
  Combining this information, we obtain
  \begin{align*}
  {\rm mod}(T:S)
  &=\dim S- \dim T +
  \dim T_0\\
 & \leqslant  \dim Y-\dim T_0-\dim T+\dim T_0\\
 &={\rm mod}(T:Y^{\rm reg}).
 \end{align*}
This completes the proof.\end{proof}

\section{\bf 7.\;Main result: proof}
We shall prove the implications
(M)$\Rightarrow$(F)$\Rightarrow$(G)$\Rightarrow$(M).

\vskip 1mm

$1$.\;The implication  (M)$\Rightarrow$(F) is clear.

$2$.\;We now turn to the proof of the implication (F)$\Rightarrow$(G).

Assume that $G$ enjoys property (F).\;Let $R$ be the radical of $G$.\;Since the group $G/R$ is semisimple, our assumption, Lemma \ref{reduction}, and Proposition \ref{ss} entail that $G/R$ is trivial, i.e.,
$G$ is solvable.\;Whence $G=T\ltimes U$, where $T$ is a maximal torus and $U$ is the unipotent radical of $G$.
We should show that
either $U$ is trivial
or $U$ is isomorphic to ${\mathbf G}_a$ and $G$ is commutative.\;Arguing
on the contrary, we suppose that
    this is not so.

    Then $U$ is a nontrivial  unipotent group.\;Hence  there exists
a chain $\{e\}=U_1\varsubsetneq U_2\varsubsetneq\cdots \varsubsetneq U_{d}=U$ of closed connected subgroups, normal in $G$, such that $d\geqslant 2$, and the successive quotients are
one-dimensional; see \cite[10.6]{B91}.

We claim that $d=2$.\;Indeed, if this is not the case, the above chain contains $U_3$.\;Since
$\dim U_3=2$,
 Proposition
\ref{2u}
and Lemma \ref{reduction} yield contra\-diction with the fact that
$G$ enjoys property
(F). Thus $d=2$; whence $U$ is isomorphic to\;${\mathbf G}_a$.

Next, the assumption that $G$ is not commutative
means that
the conju\-gat\-ing action of $T$ on $U$ is nontrivial.\;As $T$ is generated by its one-dimensional subtori, there is such a subtorus  $T'$
not lying in the kernel of this action.\;Then
$T'U$
is a noncommutative closed connected two-dimensional subgroup of $G$; see \cite[2.2]{B91}.\;Hence it is isomorphic to
$S(n)$ for some $n\neq 0$; see case (S3) in Section 5.\;By Proposition
\ref{S(n)neq0} and Lemma \ref{reduction}, this is impossible since $G$ enjoys property
(F).\;This contradiction proves the implica\-tion (F)$\Rightarrow$(G).

\vskip 1mm

$3$.\;Now we turn to the proof of the last implication (G)$\Rightarrow$(M).

Assume that (G) holds and $G$ acts  on an irreducible variety $X$.\;We should show that this action is modality-regular.\;In view of
Lemma\;\ref{torus}(ii), we should consider only the case, where
$G$ is
a direct product of two sub\-groups:
   \begin{equation}\label{TtimesU}
   G=T\times U,\quad \mbox{$T$ is a torus, $U$ is isomorphic to ${\mathbf G}_a$.}
   \end{equation}
 Below, exploring the actions of the subgroups of $\,G$ on $X$, we always mean
 the actions obtained by restricting
the given action of $\,G$ on $X$.

We may (and shall) assume that $G$ acts on $X$ faithfully.\;In view of
Corollary \ref{normal}(ii), replacing $X$ by its normalization,
we also may (and shall) assume that $X$ is normal.

Notice that
 since the elements of $T$ (respectively, $U$) are semisimple (respec\-ti\-vely,
  unipo\-tent), and the $G$-stabilizers
  of points
  of $X$, being closed in $G$, contain the Jordan decomposi\-tion com\-po\-nents of
 their ele\-ments,  for these stabilizers we have,
in view of\;\eqref{TtimesU}, the equalities:
 \begin{equation}\label{stabi}
 G_x=T_x\times U_x
 \quad\mbox{for every $x\in X$}.
 \end{equation}
As $G$ acts on $X$ faithfully, from \eqref{stabi}, \eqref{TtimesU}, and Lemma \ref{torus}(i) we infer that
\begin{equation}\label{stabili}
\mbox{$G_x$ is finite for every $x\in X^{\rm reg}$.}
\end{equation}

Let $S$ be a sheet of the action of $T$ on $X$.\;As $T$ and $U$ commute and both are connected, $S$ is $U$-stable and every sheet $C$ of the action of $U$ on $S$ is $T$-stable, hence $G$-stable.\;Consider the set of all $C$'s, obtained in this way when $S$ runs over all sheets of the action of
$T$ on $X$.\;This set is finite; we fix a numbering  of its elements: $C_1,\ldots, C_n$.\;The construction and
the condition
$\dim U=1$ yield the following:
\begin{enumerate}[\hskip 5mm\rm(i)]
\item[\rm(C1)] $X=C_1\cup\ldots\cup C_n$;
\item[\rm(C2)] every $C_i$ is a locally closed irreducible $G$-stable subset of $X$;
\item[\rm(C3)]  all $T$-orbits in $C_i$ have the same dimension $d_i$ for every $i$;
\item[\rm(C4)] for every $i$,
either $C_i^U=C_i$ or $\dim U\cd x=1$ for all $x\in C_i$.
\end{enumerate}

The construction implies that $X^{\rm reg}$ is one of these subsets; we assume that
\begin{equation}\label{XC}
X^{\rm reg}=C_1.
\end{equation}
In view of (C3) and \eqref{stabili}, we have
\begin{equation}\label{2dgp}
   \begin{split}
   {\rm mod}(T:C_i)&=\dim C_i-d_i\;\;\mbox{for every $i$},\\[-.3mm]
   d_1&=\dim T.
\end{split}
   \end{equation}
By Lemma \ref{torus}(ii), the action of $T$ on $X$ is modality-regular, so \eqref{2dgp} yields
   \begin{equation}\label{2mrt}
   \dim X-\dim T\geqslant \dim C_i-d_i \;\;\mbox{for every $i$}.
   \end{equation}

From \eqref{stabi}, (C3), (C4), we deduce that
\begin{equation}\label{2mdGS}
{\rm mod}(G:C_i)=\begin{cases}
\dim C_i-d_i&\mbox{if $C_i^U=C_i$ },\\
\dim C_i-d_i-1&\mbox{if
$C_i^U=\varnothing$.}
\end{cases}
\end{equation}
In view of \eqref{TtimesU}, \eqref{stabili}, we have
\begin{equation}\label{2fin}
{\rm mod}(G:X^{\rm reg})= \dim X-\dim T-1.
\end{equation}

Arguing on the contrary, we now suppose that the action of $G$ on $X$
is not modality-regular, i.e.,
\begin{equation}\label{2contr}
{\rm mod}(G:X)>{\rm mod}(G:X^{\rm reg}).
\end{equation}
Then, as a first step, we shall find a certain $C_{i_0}$ that has some special properties.
The next step will be analysing these properties which even\-tually will lead us to a sought-for contradiction.

Namely, by \eqref{2contr} and Lemma \ref{SI}, there is $i_0$ such that
\begin{equation}\label{2se}
{\rm mod}(G:X^{\rm reg})<{\rm mod}(G:C_{i_0}).
\end{equation}
Combining \eqref{2mrt}, \eqref{2mdGS}, \eqref{2fin}, \eqref{2se}, we obtain
\begin{equation}\label{2ineqs}
\begin{split}
\dim C_{i_0}-d_{i_0}-1&\overset{\rm\eqref{2mrt}}\leqslant \dim X-\dim T-1\\
&\overset{\eqref{2fin}}{=\hskip -1.4mm=}{\rm mod}(G:X^{\rm reg})\\
&\overset{\rm\eqref{2se}}{<} {\rm mod}(G:C_{i_0})\\
&\overset{\rm\eqref{2mdGS}}{=\hskip -1.4mm=}\begin{cases}
\dim C_{i_0}-d_{i_0}&\mbox{if $C_{i_0}^U=C_{i_0}$},\\
\dim C_{i_0}-d_{i_0}-1&\mbox{if $C_{i_0}^U=\varnothing$.}
\end{cases}
\end{split}
\end{equation}
In turn, from \eqref{2ineqs} we infer the following:
\begin{align}\label{2*1}
C_{i_0}^U&= C_{i_0},\\
    \label{2**}
    \dim C_{i_0}-d_{i_0}&= \dim X-\dim T.
    \end{align}

    Denote by $T_{i_0}$ be the identity component of the kernel of  the action of $\,T$ on $C_{i_0}$ and consider in $G$ the closed subgroup
\begin{equation}\label{H}
H:=T_{i_0}\times U.
\end{equation}
By (C3) and Lemma \ref{torus}(i), we have
\begin{align}\label{kerT1}
\dim T_{i_0}&=\dim T-d_{i_0},\\
\dim H&=\dim T-d_{i_0}+1.\label{kerT2}
\end{align}
From \eqref{2*1} and the definitions of $T_{i_0}$ and $H$ we infer that
\begin{equation}\label{S0fixed}
C_{i_0}^H=C_{i_0}.
\end{equation}

By \cite[Cor.\;2, p.\;8]{S74}, as $X$ is normal,
it is covered by the $T_{i_0}$-stable
  affine open subsets.\;Whence there is a $T_{i_0}$-stable
  affine open subset $A$ in $X$ such that
  \begin{align}\label{AS0}
  &\mbox{$A\cap C_{i_0}$ is a dense open subset of $C_{i_0}$,}\\
  &\mbox{$A\cap X^{\rm reg}$ is a dense open subset of $A$.}\label{AS02}
  \end{align}

  Consider the categorical quotient for the affine $T_{i_0}$-variety $A$:
  \begin{equation*}
  \pi\colon A\to A/\!\!/T_{i_0}=:{\rm Spec}\,k[A]^{T_{i_0}}.
  \end{equation*}
 By \eqref{stabili}, we have $\dim T_{i_0}\cd x=\dim T_{i_0}$ for every $x\in X^{\rm reg}$.
 This, the fiber dimension theorem,
 the $T_{i_0}$-equivariance of $\pi$,
 and the equality $\dim A=\dim X$ then yield:
\begin{equation}\label{leAT0}
\dim A/\!\!/T_{i_0}\leqslant
\dim A-\dim T_{i_0}
\overset{\rm \eqref{kerT1}}{=\hskip -1.4mm=}\dim X-\dim T+d_{C_{i_0}}
\overset{\rm \eqref{2**}}{=\hskip -1.4mm=}\dim C_{i_0}.
\end{equation}

On the other hand, since $k[A]^{T_{i_0}}$ separates disjoint closed $T_{i_0}$-stable sub\-sets of $A$ (see \cite[Thm.\;9.4]{PV94}),  we have
\begin{equation}\label{geAT0}
\dim C_{i_0}\overset{\rm \eqref{S0fixed}}{=\hskip -1.4mm=}\dim \pi (C_{i_0})\leqslant \dim A/\!\!/T_{i_0}
\end{equation}

From \eqref{leAT0}, \eqref{geAT0} we obtain the equalities
\begin{equation}\label{dimAT0S0}
\dim C_{i_0}=\dim A/\!\!/T_{i_0}=\dim A-\dim T_{i_0}.
\end{equation}
In turn, from \eqref{dimAT0S0}, \eqref{AS0},
\eqref{AS02}, and the fiber dimension theorem,  we
deduce the existence of
a dense open subset $Q$
of $A/\!\!/T_{i_0}$
that enjoys  the following proper\-ties:
\begin{align}\label{Qin}
  &Q\subseteq \pi(A\cap C_{i_0})\cap \pi (A\cap X^{\rm reg}),\\
  &\mbox{$\pi^{-1}(q)$ is equidimensional of dimension $\dim T_{i_0}$ for every $q\in Q$.}\label{fiberdim}
\end{align}

Now take a point $x\in \pi^{-1}(Q)\cap X^{\rm reg}$.\;In view of \eqref{stabili},
we have
\begin{equation}\label{irreco}
\dim T_{i_0}\cd x=\dim T_{i_0}.
\end{equation}
As orbits are open in their closures, and $T_{i_0}\cd x\subseteq \pi^{-1}(\pi(x))$,
from \eqref{fiberdim}, \eqref{irreco}
we infer that $T_{i_0}\cd x$ is a dense open subset of an irreducible compo\-nent of
the fiber $\pi^{-1}(\pi(x))$.\;In view of \eqref{Qin},  this fiber contains a point $s\in C_{i_0}$, so we have
\begin{equation}\label{xs}
\pi^{-1}(\pi(x))=\pi^{-1}(\pi(s)).
\end{equation}

As,  by \eqref{S0fixed}, the point $s$ is $T_{i_0}$-fixed, it lies in the closure of  $\,T_{i_0}\cd x$ in $A$ (and a fortiori in $X$); see \cite[Thm.\;4.7]{PV94}.\;Thus $\,T_{i_0}\cd x$ belongs to the set $\mathscr S$ of all $T_{i_0}$-orbits
$\mathcal O$ in $X$
that enjoy the following properties:
 \begin{enumerate}[\hskip 4.2mm\rm(a)]
 \item $\dim \mathcal O=\dim T_{i_0}$;
 \item the closure $\overline {\mathcal O}$ of $\mathcal O$ in $X$ contains $s$.
 \end{enumerate}

 We claim that $\mathscr S$ is finite.\;Indeed, if
 a $T_{i_0}$-orbit
 $\mathcal O$ belongs to $\mathscr S$, then $\overline {\mathcal O}\cap A$ is an open
 neighbourhood of $s$ in $\overline {\mathcal O}$, therefore $\mathcal O\cap A\neq \varnothing$.
 Whence $\mathcal O$ lies in $A$ and contains $s$ in its closure in $A$.\;This and \eqref{xs}  show that $\mathcal O$ is a $\dim T_{i_0}$-dimensional $T_{i_0}$-orbit of $\pi^{-1}(\pi(x))$; whence, as above, $\mathcal O$ is a dense open subset of an irreducible component of $\pi^{-1}(\pi(x))$.\;The claim now follows from the finiteness of the set of irreducible components of
 $\pi^{-1}(\pi(x))$.

The finiteness of  $\mathscr S$ implies that the union of all $T_{i_0}$-orbits from $\mathscr S$ is a locally closed subset $Z$ of $X$ whose irreducible components are these orbits.\;As we proved above, one of these components is $T_{i_0}\cd x$.\;Since $U$ commutes with $T_{i_0}$ and, by \eqref{2*1}, $s$ is a $U$-fixed point, the subset $Z$ is $U$-stable.\;The connectedness of $U$ then entails that each irreducible component of this subset is $U$-stable.\;In par\-ti\-cu\-lar, $T_{i_0}\cd x$ is $U$-stable. Whence $T_{i_0}\cd x$ is $H$-stable and therefore we have
\begin{equation}\label{HT0orb}
H\cd x=T_{i_0}\cd x.
\end{equation}

In view of  \eqref{stabili}, \eqref{H}, \eqref{TtimesU},
we now obtain
the sought-for contradiction:
 \begin{equation*}
 \dim T_{i_0}+1=\dim H=\dim H\cd x\overset{\rm \eqref{HT0orb}}{=\hskip -1.4mm=}\dim T_{i_0}\cd x=\dim T_{i_0}.
 \end{equation*}
This completes the proof.\hfill $\square$

\end{document}